\documentclass{article}
\usepackage{maa-monthly}

\def\N{\mathbb{N}}
\def\Z{\mathbb{Z}}

\theoremstyle{theorem}
\newtheorem{theorem}{Theorem}
\newtheorem{lemma}[theorem]{Lemma}
 
\theoremstyle{definition}

\begin{document}

\title{Are the Catalan Numbers a Linear Recurrence Sequence?}
\markright{The Catalan numbers}
\author{Martin Klazar and Richard Horsk\'y}

\maketitle

\begin{abstract}
We answer the question in the title in the negative by providing four proofs.
\end{abstract}

\section{Introduction.}

Let $\mathbb{N}=\{1,2,\dots\}$ be the set of natural numbers, $\mathbb{N}_0=\{0,1,2,\dots\}$ be the nonnegative integers, $\Z$ be all
integers, $\mathbb{Q}$ be the field 
of fractions, and $\mathbb{C}$ be the field of complex numbers. For $n\in\N$ we set $[n]=\{1,2,\dots,n\}$. The Catalan numbers
$$
(C_n)_{n\ge1}=(1,\,1,\,2,\,5,\,14,\,42,\,132,\,429,\,1430,\,4862,\,16796,\,58786,\,\dots)
$$
are among the most popular sequences in enumerative combinatorics. In \cite{stan_CN}, R.~P.~Stanley gives more than
$200$ mostly combinatorial 
interpretations of the sequence $(C_n)$. What would be their simplest self-contained combinatorial definition? Perhaps, for $n\in\mathbb{N}$ 
and $n\ge2$, 
$${\textstyle
C_n=\big|\{\overline{a}\in\{-1,1\}^{2n-2}\;|\;\sum_{j=1}^k a_j\ge0\text{ for }k\in[2n-3]\text{ and }\sum_{j=1}^{2n-2}a_j=0\}\big|\;.
}
$$
Here $\overline{a}=(a_1,a_2,\dots,a_{2n-2})$ and $|X|$ denotes cardinality of a finite set $X$. Thus $C_n$ counts the words built from $n-1$ 
ones and $n-1$ minus ones and with all initial sums nonnegative. 
Starting from this definition, it is not hard to deduce the recurrence $C_1=1$ and for $n\ge2$,
$$
C_n=\sum_{j=1}^{n-1}C_jC_{n-j}
$$
(consider the initial segments with sums equal to $1$)
and also the closed formula $C_n=\frac{1}{n}\binom{2n-2}{n-1}$. From it we easily derive another recurrence, namely $C_1=1$ and 
for $n\ge1$,
$$
C_{n+1}=\frac{4n-2}{n+1}\cdot C_n\;.
$$
Could it be that the $C_n$ follow a yet simpler recurrence that for every $n\ge1$,
\begin{equation}\label{one}
    C_{n+k}=\sum_{j=0}^{k-1} a_jC_{n+j}
\end{equation}
where $k\in\mathbb{N}_0$ and the $a_j\in\mathbb{C}$ are constants? For $k=0$ the sum is defined as $0$. As Theorem~\ref{field_exte} in Section~\ref{sec_NT} 
shows, without loss of generality we  may take $a_j\in\mathbb{Q}$.

\begin{theorem}
The answer is negative, $(C_n)$ is not a \emph{linear recurrence sequence}
as the numbers $C_n$ do not satisfy any recurrence of the form $(\ref{one})$.
\end{theorem}
In the following four sections we give four proofs. \emph{Linear recurrence sequences} are sequences 
$$
(b_n)=(b_1,\,b_2,\dots)\subset\mathbb{C}
$$ 
(or with elements lying in another field) satisfying for every $n\in\mathbb{N}$ a recurrence of the form (\ref{one}). They appear frequently 
in enumerative combinatorics and number theory and 
the monograph \cite{ever} is devoted to them. One of the best known examples are the Fibonacci numbers
$(F_n)=(1,1,2,3,5,8,13,21,\dots)$, satisfying for every $n\in\mathbb{N}$ the recurrence $F_{n+2}=F_{n+1}+F_n$. It is our (immodest) goal that due to our article the Catalan numbers become one of the best known nonexamples. 

\section{A proof by asymptotics.}\label{sec_asym}

The first, if not always the best, thought of an enumerative combinatorialist facing the task of showing that a sequence of integers is not a
linear recurrence sequence $(b_n)\subset\mathbb{Z}$ is to use the asymptotics of the sequence. One might say (we warn the reader that this argument is not rigorous) that since 
$$
C_n\sim c n^{-3/2}4^n\;\text{ and }\;b_n\sim d n^s\alpha^n\;\text{ as }\;n\to\infty\;,
$$
where $c>0$ is a real constant, $s\in\mathbb{N}_0$, and $d,\alpha>0$ are real algebraic numbers, $(C_n)$ is not a linear 
recurrence sequence because the two asymptotics are incompatible. The problem with this argument is that while the former asymptotic expression is correct (it follows for example 
from the closed formula for $C_n$ and the Stirling 
asymptotics for $n!$), the latter asymptotic expression is in general incorrect. It only holds 
when the expression for $b_n$ by a power sum (which we recall below) has a~so-called dominant root, a unique root with the maximum modulus $\alpha$. For linear recurrence sequences 
arising in enumerative combinatorics (for example, for the Fibonacci numbers) this is usually the case, but for a general linear recurrence sequence its 
power sum need not have dominant root. 

The {\em power sum} representation of a linear recurrence sequence $(b_n)\subset\mathbb{C}$, which generalizes the well-known Binet formula for the Fibonacci numbers $F_n$, is an expression of the form ($n\in\mathbb{N}$)
$$
b_n=\sum_{j=1}^r p_j(n)\alpha_j^n\;,
$$
where $r\in\mathbb{N}_0$, $p_j\in\mathbb{C}[x]$ are nonzero polynomials and $\alpha_j\in\mathbb{C}$ are mutually distinct nonzero numbers, the so-called roots of the power sum; for $r=0$ the sum 
is defined as $0$. For a proof that $b_n$ has such an expression see, for example, \cite[Chapter 4]{stan_EC1}. Thus the correct asymptotics for $b_n$ is 
$$
b_n=n^s\alpha^n\sum_{j=1}^l\gamma_j\beta_j^n+O(n^{s-1}\alpha^n)\;\text{ for }\;n\in\mathbb{N}\;,
$$
where $s\in\mathbb{N}_0$, $\alpha>0$ is the maximum modulus $|\alpha_j|$ of a root $\alpha_j$, $l\in\mathbb{N}_0$, $\gamma_j\in\mathbb{C}$ are nonzero numbers, and
$\beta_j\in\mathbb{C}$ are mutually distinct numbers with $|\beta_j|=1$.

Let $l>0$. For $n\in\mathbb{N}$ we define
$$
v(n)=\sum_{j=1}^l\gamma_j\beta_j^n
$$
and prove that $\limsup_{n\to\infty}|v(n)|>0$. Together with the trivial upper bound $|v(n)|=O(1)$, this saves the argument above and shows that indeed 
$C_n\sim cn^{-3/2}4^n$ cannot be expressed by a power sum and therefore $(C_n)$ is not a linear recurrence sequence. 

Suppose by way of contradiction that $\lim_{n\to\infty}v(n)=0$. This in particular means that for every $k\in\mathbb{N}$ there is an $n_k\in\mathbb{N}$ such that all $l$
values $|v(n_k+1)|,|v(n_k+2)|,\dots,|v(n_k+l)|$ are at most $1/k$. We solve the $l$ linear equations
$$
v(n_k+i)=\sum_{j=1}^l\gamma_j\beta_j^{n_k+i},\ i=1,\,2,\,\dots,\,l\;,
$$
for the $\gamma_j$. By Cramer's rule we get the formulas
$$
\gamma_j=\frac{\det(M_{j,k})}{\det\left(\left(\beta_m^{n_k+i}\right)_{i,m=1}^l\right)}=:\frac{A_{j,k}}{B_k},\ j=1,\,2,\,\dots,\,l\;,
$$
where the matrix $M_{j,k}$ arises from the matrix in the denominator by replacing the $j$th column with the column $(v(n_k+1)\;\cdots\;v(n_k+l))^T$.
It follows from the bounds $|\beta_m|=1$ and $|v(n_k+i)|\le 1/k$ and the definition of determinant that $|A_{j,k}|\le l!/k$. We take the power $\beta_m^{n_k+1}$ out of the $m$th column of the matrix in the denominator and obtain a Vandermonde matrix. The formula for the Vandermonde determinant shows that
$$
|B_k|=\prod_{1\le u<v\le l}|\beta_v-\beta_u|=:\beta>0\;,
$$
independently on $k$.
Thus $|\gamma_j|\le l!/\beta k$ for $j=1,2,\dots,l$. Letting $k$ go to $\infty$, we get that all $\gamma_j=0$, in contradiction with the assumption. 

\section{A proof by generating functions.}\label{sec_gf}

This is the shortest of the four proofs, at least when we take for granted that the $C_n$ have the generating function
$$
C(x)=\sum_{n=1}^{\infty}C_nx^n=\frac{1-\sqrt{1-4x}}{2}\;,
$$
and that the generating function of every linear recurrence sequence $(b_n)\subset\mathbb{C}$ is rational, i.e.,
$$
B(x)=\sum_{n=1}^{\infty}b_nx^n=\frac{p(x)}{q(x)}\;\text{ for some polynomials $p,q\in\mathbb{C}[x]$ with $q(0)\ne0$}\;.
$$
The formula for $C(x)$ follows easily from the relation $C_n=\sum_{j=1}^{n-1}C_jC_{n-j}$. For a~proof of rationality of $B(x)$, we refer again to
 \cite[Chapter 4]{stan_EC1}. We work in the ring $\mathbb{C}[[x]]$ of formal power series; thus $\sqrt{1-4x}$ means $\sum_{n=0}^{\infty}\binom{1/2}{n}(-4)^nx^n$
 and $\frac{p(x)}{q(x)}$ means $p(x)q(x)^{-1}$ where $q(x)^{-1}$ is the multiplicative inverse of $q(x)$ in $\mathbb{C}[[x]]$.
 
If $(C_n)$ were a linear recurrence sequence, then $C(x)$ would be rational and we would have the equality
 $$
 \sqrt{1-4x}=\frac{a(x)}{b(x)}\;\text{ for some polynomials $a,b\in\mathbb{C}[x]$ with $b(0)\ne0$}\;.
 $$
 Also, $a(x)\ne0$. Hence $b(x)^2(1-4x)=a(x)^2$. This is impossible, as the left-hand side has an odd degree but the right-hand side has an even degree. This argument resembles the well-known proof of the irrationality of $\sqrt{2}$ taught in courses 
of mathematical analysis. This proof may be the most folkloric one of the four since irrationality of $C(x)$ is well known; see 
for example the course notes \cite[p.~1]{pak}.

\section{A proof by number theory.}\label{sec_NT}

To fulfill our promise, before we give the third proof, we prove a result relating the domain of terms of a linear recurrence 
sequence and the domain of coefficients in the recurrence. The result is well known, at least to the researchers interested 
in linear recurrence sequences, but its proof cannot be found easily in the literature, and thus it may be of some interest to present it here. We need 
the next standard lemma from linear algebra (or perhaps from number theory? see Siegel's lemma), whose proof we leave to the reader as an exercise.

\begin{lemma}
Let $K$ be a field, $m,n\in\mathbb{N}$ with $m<n$, and $a_{i,j}\in K$ for $i=1,2,\dots,m$ and $j=1,2,\dots,n$ 
be $mn$ elements. Then there exist $n$ elements $x_1,x_2,\dots, x_n$ 
in $K$, not all equal to $0_K$, such that for every $i=1,2,\dots,m$, 
$$
\sum_{j=1}^n a_{i,j}x_j=0_K\;.
$$
\end{lemma}

\begin{theorem}\label{field_exte}
Suppose that $K\subset L$ is an extension of fields and $(b_n)\subset K$ is a linear recurrence sequence given by a recurrence $$b_{n+k}=\sum_{j=0}^{k-1}a_jb_{n+j},\ n\in\mathbb{N}\;,$$ 
with $k\in\mathbb{N}_0$ and all coefficients $a_j\in L$. Then $(b_n)$ satisfies another recurrence 
$$b_{n+k'}=\sum_{j=0}^{k'-1}a_j'b_{n+j},\ n\in\mathbb{N}\;,$$
with $k'\in\mathbb{N}_0$, $k'\le k$, and all coefficients $a_j'\in K$.
\end{theorem}

\begin{proof}
For every $n\in\mathbb{N}$,
$$
\sum_{j=0}^k a_jb_{n+j}=0_L\;\text{ where }\;a_k:=-1_L\;.
$$
We set 
$$
B=\{\overline{b}=(b_n,\,b_{n+1},\,\dots,\,b_{n+k})\;|\;n\in\mathbb{N}\}\subset L^{k+1}\;\text{ and }\;
d=\dim_L(B)\in\mathbb{N}_0\;,
$$ 
where we understand $L^{k+1}$ to be a vector space over $L$. Since for every tuple $\overline{b}\in B$ the nonzero 
$(k+1)$-tuple $\overline{a}=(a_0,a_1,\dots,a_k)$ 
satisfies $\langle\overline{a},\overline{b}\rangle=0_L$ (scalar product), we have $d<k+1$. Let $B_0\subset B$ be a maximal linearly independent subset; thus 
$\dim_L(B_0)=|B_0|=d$ and every tuple $\overline{b}\in B$ is an $L$-linear combination of the tuples in $B_0$. 
Since $B_0\subset K^{k+1}$, by the 
previous lemma there exists a nonzero vector $\overline{a^*}\in K^{k+1}$ such that 
$\langle\overline{a^*},\overline{b}\rangle=0_K$ for every $\overline{b}\in B_0$. Every tuple in $B$ is a linear combination 
of the tuples in $B_0$, and therefore we have that
$\langle\overline{a^*},\overline{b}\rangle=0_K$ even for every $\overline{b}\in B$.
Let $\overline{a^*}=(a_0^*,a_1^*,\dots,a_k^*)$ and $k'\in\mathbb{N}_0$ be maximum with $a_{k'}^*\ne0_K$. 
For $j=0,1,\dots,k'-1$ we set $a_j'=-a_j^*/a_{k'}^*\in K$. Then for every $n\in\mathbb{N}$,
$$
b_{n+k'}=\sum_{j=0}^{k'-1}a_j'b_{n+j}\;.
$$
\end{proof}

Thus by Theorem~\ref{field_exte}
for any linear recurrence sequence of integers we may assume that the recurrence coefficients are fractions. 

We proceed to the third proof and assume for contradiction that for some $k$ in $\mathbb{N}_0$ and $a_0,a_1,\dots,a_{k-1}$ in $\mathbb{Q}$, $a_k=-1$, the relation
\begin{equation}\label{two}
    \sum_{j=0}^k a_j C_{n+j}=0
\end{equation}
holds for every $n\in\mathbb{N}$. 
We make use of the fact that the odd $C_n$ are infinite in number but increasingly isolated. The reader can see that of the numbers $C_1,C_2,\dots, C_{12}$ given 
in the Introduction, odd ones are $C_1,C_2,C_4$, and $C_8$. Indeed, in general
$$
\text{$C_n$ is odd if and only if $n=2^m$ for some $m\in\mathbb{N}_0$}\;.
$$
This follows easily from the relation $C_n=\sum_{j=1}^{n-1}C_jC_{n-j}$. Indeed, it shows that if $n>1$ is odd then $C_n$ is even (but $C_1=1$), and that if $n$ is even then 
$C_n\equiv C_{n/2}^2\equiv C_{n/2}$ modulo $2$. The equivalence now follows by writing an $n\in\mathbb{N}$ as $n=2^m r$ with $m\in\mathbb{N}_0$ and odd
$r\in\mathbb{N}$. This result on parity of $C_n$ is certainly well known but we do not know any reference to its origin. 

Multiplying the above relation (\ref{two}) by an appropriate natural number, we clear the denominators and assume without loss of generality that the numbers $a_0,a_1,\dots,a_k$ are mutually coprime integers. In particular, some coefficient $a_l$ must be odd. We take a large enough $n\in\mathbb{N}$ such that
$n+l$ is the sole power of $2$ among $n,n+1,\dots,n+k$. The above relation is then impossible because the sum contains exactly one odd summand, $a_lC_{n+l}$, all other summands are even, and the sum should equal to $0$.

\section{A proof by polynomials.}\label{sec_poly}

We thought of this proof as the last one but in retrospect it is quite natural. We assume for contradiction that for
some $k\in\mathbb{N}_0$, $a_0,a_1,\dots,a_{k-1}\in\mathbb{C}$, and $a_k=-1$, the relation
$$
\sum_{j=0}^k a_jC_{n+j}=0
$$
holds for every $n\in\mathbb{N}$. We substitute for each $C_{n+j}$ the formula $\frac{1}{n+j}\binom{2n+2j-2}{n+j-1}$
and after simplification get that $p(n)=0$ for a nonzero polynomial $p\in\mathbb{C}[x]$ and every $n\in\mathbb{N}$, which is impossible. 
This may be viewed as a very simple instance of the polynomial methods exposed in L.~Guth's book \cite{guth}. 

In more detail, for a variable $x$ and $k\in\mathbb{N}_0$ we consider the descending products $(x)_0:=1$ and for $k>0$,
$$
(x)_k:=x(x-1)(x-2)\cdots(x-k+1)\;.
$$
It is a monic polynomial in $x$ with degree $k$. As we said, in the above displayed relation we set 
$$
C_{n+j}=\frac{1}{n+j}\cdot\frac{(2n+2j-2)!}{(n+j-1)!^2}
$$ 
and, to get rid of common factors and denominators, multiply it by 
$$
(n+k)_{k+1}\cdot\frac{(n+k-1)!^2}{(2n-2)!}\;. 
$$
From this we get the relation
$$
p(n):=\sum_{j=0}^k a_j\cdot\frac{(n+k)_{k+1}}{n+j}\cdot(n+k-1)_{k-j}^2\cdot(2n+2j-2)_{2j}=0
$$
holding for every natural number $n$. The remaining crucial step is to show
that $p(x)$ is not the zero polynomial. This is not so obvious because each of the $k+1$ summands (without $a_j$) is a nonzero polynomial in $n$ with 
the same degree $3k$ and a~cancellation might occur. But the evaluation at $x=n:=-k$ shows that
$$
p(-k)=(-1)\cdot(-1)_k\cdot1\cdot(-2)_{2k}\ne0
$$
because  all other terms vanish except for $j=k$.
Thus $p(x)\ne0$ and we get the contradiction showing that $(C_n)$ is not a linear recurrence sequence.

\begin{acknowledgment}{Acknowledgment.}Our thanks go to the two referees and also to the editor, for their careful reading 
of our manuscript and for the many improvements they suggested. 
\end{acknowledgment}

\begin{biog}
\item [Martin Klazar] teaches mathematics to students of computer science at Charles University in Prague. He enjoys walks, 
sometimes even jogging, on hill V\'\i tkov near his residence. \begin{affil} Charles University,
Faculty of Mathematics and Physics,
Department of Applied Mathematics,
Malostransk\'e n\'am. 25,
118 00 Praha 1,
Czech Republic\\ klazar@kam.mff.cuni.cz
\end{affil}

\item[Richard Horsk\'y] teaches mathematics at Prague University of Economics and Business. 
He enjoys random walks in forests near his house in B\v rezov\'a-Ole\v sko. \begin{affil} Prague University of Economics 
and Business, Faculty of Informatics and Statistics, Department of Mathematics, Ekonomick\'a 957,
148 00 Praha 4-Kunratice, Czech Republic\\
rhorsky@vse.cz
\end{affil}
\end{biog}

\end{document}